\documentclass[12pt]{amsart}
\usepackage{a4wide}

\newcommand{\IR}{\mathbb R}
\newcommand{\IZ}{\mathbb Z}
\newcommand{\IT}{\mathbb T}
\newcommand{\C}{\mathcal C}

\newcommand{\w}{\omega}

\newtheorem{theorem}{Theorem}[section]
\newtheorem{problem}[theorem]{Problem}
\newtheorem{proposition}[theorem]{Proposition}
\newtheorem{lemma}[theorem]{Lemma}

\newtheorem{example}[theorem]{Example}

\title{Weak completions of paratopological groups}

\dedicatory{This article is respectfully dedicated to Professor Jerzy Mioduszewski on the occasion of his 93rd anniversary}

\author{Taras Banakh}

\address{T.Banakh: Ivan Franko National University of Lviv (Ukraine) and 
Jan Kochanowski University in Kielce (Poland)}
\email{t.o.banakh@gmail.com}
\author{Mikhail Tkachenko}

\address{M.Tkachenko: Department of Mathematics, Universidad Aut\'onoma Metropolitana,
Av. San Rafael Atlixco 186, Col. Vicentina, C.P. 09340, Iztapalapa, Mexico City, Mexico}
\email{mich@xanum.uam.mx}


\thanks{The second author was supported by grant number CAR-64356 of the Program \lq\lq{Ciencia de Frontera 2019\rq\rq} of the CONACyT, Mexico}

\begin{document}
\begin{abstract} 
Given a $T_0$ paratopological group $G$ and a class $\C$ of continuous homomorphisms of 
paratopological groups, we define the \emph{$\C$-semicompletion} $\C[G)$ and 
\emph{$\C$-completion} $\C[G]$ of the group $G$ that contain $G$ as a dense subgroup, satisfy the  
$T_0$-separation axiom 
and have certain universality properties. For special classes $\C$, we present some necessary and sufficient conditions on $G$ in order that the (semi)completions $\C[G)$ and  $\C[G]$ be Hausdorff.  Also, we give
an example of a Hausdorff paratopological abelian group $G$ whose $\C$-semicompletion $\C[G)$ fails 
to be a $T_1$-space, where $\C$ is the class of continuous homomorphisms of sequentially compact
topological groups to paratopological groups. In particular, the group $G$ contains an $\omega$-bounded 
sequentially compact subgroup $H$ such that $H$ is a topological group but its closure in $G$ fails to be 
a subgroup.
\end{abstract}
\maketitle

\section{Introduction}

It is well known \cite[\S3.6]{AT} that each topological group $G$ has the Ra\u{\i}kov completion, 
$\varrho{G}$, which coincides with the completion of $G$ with respect to the two-sided group 
uniformity of the group. The Ra\u{\i}kov completion has a nice categorial characterization. It turns 
out that $\varrho{G}$ is a unique topological group, up to topological isomorphism, that contains 
$G$ as a dense subgroup and has the following extension property: {\em Every continuous 
homomorphism $h\colon X\to G$ defined on a dense subgroup $X$ of a topological group 
$\widetilde{X}$ admits a unique extension to a continuous homomorphism $\tilde{h}\colon
\widetilde{X}\to\varrho{G}$} (see \cite[Proposition~3.6.12]{AT}). A topological group $G$ is 
{\em Ra\u{\i}kov complete} if and only if it is complete in its two-sided uniformity if and only 
if $G=\varrho{G}$.

In \cite{MR} and \cite{BR}, it is shown that a kind of the Ra\u{\i}kov completion can also be 
defined for paratopological groups. By a \emph{paratopological group} we understand 
a group $G$ endowed with a topology $\tau$ making the group multiplication $G\times G\to G$, 
$\langle x,y\rangle\to xy$ continuous. If, in addition, the inversion $G\to G$, $x\mapsto x^{-1}$, 
is continuous, then $\langle G,\tau\rangle$ is a topological group. A topology $\tau$ on a group 
$G$ is called a {\em topological group topology} on $G$ if $\langle G,\tau\rangle$ is a topological group.

Let us recall that a topology $\tau$ on a set $X$ satisfies the $T_0$ separation axiom or, equivalently, 
$\tau$ is a {\em $T_0$-topology} if for any distinct points $x,y\in G$ there exists an open set $U\in\tau$ 
such that $U\cap\{x,y\}$ is a singleton. It is well known, on the one hand, that each topological group 
satisfying the $T_0$ separation axiom  is automatically Tychonoff. On the other hand, the topology 
$\tau=\{\emptyset,\mathbb R\}\cup\{(a,\infty): a\in\IR\}$ turns the additive group of the reals into a 
$T_0$ paratopological group which does not satisfy the $T_1$ separation axiom. So \lq{$T_0$ does
not imply $T_1$\rq} in paratopological groups. From now on we assume that {\em all paratopological 
groups considered here satisfy the $T_0$ separation axiom}. 

Each paratopological group $\langle G,\tau\rangle$ admits a stronger topological group topology 
$\tau^\sharp$, whose neighborhood base at the identity $e$ consists of the set $U\cap U^{-1}$ where 
$e\in U\in\tau$. So, $\langle G,\tau^\sharp\rangle$ is topologically isomorphic to the subgroup $\{\langle x,x^{-1}\rangle: x\in G\}$ of the product group $G\times G$. The topological group $G^\sharp=\langle G,\tau^\sharp\rangle$ is called the {\em group coreflection} of the paratopological group $G=\langle G,\tau\rangle$. If the topology $\tau$ satisfies the $T_0$ separation axiom, then the topological group topology $\tau^\sharp$
is Hausdorff and hence Tychonoff. A paratopological group $\langle G,\tau\rangle$ is a topological group 
if and only if $\tau=\tau^\sharp$. A paratopological group $G$ is called {\em $\sharp$-complete} if the 
topological group $G^\sharp$ is Ra\u{\i}kov complete. The Sorgenfrey line $\mathbb{S}$ is an easy 
example of a $\sharp$-complete paratopological group since $\mathbb{S}^\sharp$ is the discrete
group of the reals.

A subset $D$ of a paratopological group $\langle G,\tau\rangle$ is called {\em $\sharp$-dense} in 
$G$ if $D$ is dense in $\langle G,\tau^\sharp\rangle$. Since $\tau\subseteq \tau^\sharp$, each 
$\sharp$-dense subset is dense (but not vice versa).

According to \cite{MR} or \cite{BR}, each paratopological group $G$ is a $\sharp$-dense subgroup of a 
$\sharp$-complete paratopological group $\breve{G}$ that has the following extension property: {\em any continuous homomorphism $h\colon X\to G$ defined on a $\sharp$-dense subgroup $X$ of a paratopological group $\widetilde{X}$ admits a unique extension to a continuous homomorphism $\tilde h\colon \widetilde{X}
\to \breve{G}$}. This extension property of $\breve{G}$ implies that the $\sharp$-complete paratopological 
group $\breve{G}$ containing $G$ as a $\sharp$-dense subgroup is unique up to a topological isomorphism. 
In \cite{BR} this unique paratopological group $\breve{G}$ is called the {\em Ra\u{\i}kov completion} of $G$ and in \cite{MR} it is called the {\em bicompletion} of $G$. By \cite{BR}, a neighborhood base at the identity $e$ of the paratopological group $\breve{G}$ consists of the sets $\overline{U}^\sharp$, where  $U$ is a neighborhood 
of $e$ in $G$ and $\overline{U}^\sharp$ is the closure of $U$ in the Ra\u{\i}kov completion $\varrho{G}^\sharp$ of the topological group $G^\sharp$. So as a group, $\breve{G}$ coincides with the Ra\u{\i}kov completion of $G^\sharp$. If $G$ is a topological group, then the Ra\u{\i}kov completion $\breve{G}$ of $G$ as a paratopological group coincides with the usual Ra\u{\i}kov completion $\varrho{G}$ of $G$.

 By \cite{BR}, \emph{for any regular paratopological group $G$, its Ra\u{\i}kov completion $\breve{G}$ 
 is regular and hence Tychonoff}  (for the latter implication, see \cite{BRT}). However, the Ra\u{\i}kov 
 completion of a Hausdorff paratopological group $G$ is not necessarily Hausdorff (see \cite[Example~2.9]{BR}). To bypass this pathology of $\breve{G}$, in this article we consider some weaker notions of completion of paratopological groups which preserve Hausdorffness in some special cases.

Our weaker notion of completion depends on a class $\mathcal{C}$ of continuous homomorphisms between paratopological groups. For example, $\mathcal{C}$ can be the class of all continuous homomorphisms  of topological groups (possessing some property like precompactness, pseudocompactness, countable compactness, sequential compactness, etc.) to paratopological groups. So, we fix a class $\mathcal{C}$ of continuous homomorphisms between paratopological groups. 

A paratopological group $\widetilde G$ is called a {\em $\mathcal{C}$-semicompletion} of a paratopological group $G$ if $G\subseteq\widetilde G$ and any homomorphism $h\colon X\to G$ in the class $\mathcal{C}$ 
has a unique extension $\breve{h}\colon \breve{X}\to\widetilde{G}$ to the Ra\u{\i}kov completion $\breve{X}$ 
of the paratopological group $X$. The extension property of the Ra\u{\i}kov completion $\breve{G}$ ensures 
that $\breve{G}$ is a $\mathcal{C}$-semicompletion of $G$. So every paratopological group  has at least 
one $\mathcal{C}$-semicompletion. 

A $\mathcal{C}$-semicompletion $\widetilde G$ of a paratopological group $G$ is called {\em minimal} if 
$\widetilde{G}=H$ for any subgroup $H\subseteq\widetilde{G}$ which is a $\mathcal{C}$-semicompletion 
of $G$. Every $\mathcal{C}$-semicompletion $\tilde G$ of a paratopological group $G$ contains a unique 
minimal $\mathcal{C}$-semicompletion of $G$, which is equal to the intersection $\bigcap \mathcal{S}$ 
of the family $\mathcal{S}$ of subgroups $H\subseteq\widetilde{G}$ such that $H$ is a $\mathcal{C}$-semicompletion of $G$. The Hausdorffness of the group coreflections implies that $\bigcap \mathcal{S}$ 
is indeed a $\mathcal{C}$-semicompletion of $G$. 

The {\em $\mathcal{C}$-semicompletion} $\mathcal{C}[G)$ of a paratopological group $G$ is defined 
to be the smallest $\mathcal{C}$-semicompletion of $G$ in its Ra\u{\i}kov completion $\breve{G}$. 
It is equal to the intersection $\bigcap\mathcal S$ of the family $\mathcal{S}$ of all groups $H$ with 
$G\subseteq H\subseteq \breve{G}$ such that $H$ is a $\mathcal{C}$-semicompletion of $G$. Let 
us note that the $\mathcal{C}$-semicompletion $\mathcal{C}[G)$ of $G$ is topologically isomorphic 
to any minimal $\mathcal{C}$-semicompletion  of $G$. Indeed, given a minimal 
$\mathcal{C}$-semicompletion $H$ of $G$, consider the Ra\u{\i}kov completion $\breve{H}$ of the paratopological group $H$. Since $G\subseteq H\subseteq\breve{H}$, the Ra\u{\i}kov completion 
$\breve{G}$ of $G$ can be identified with the $\sharp$-closure of $G$ in $\breve{H}$. Now the 
minimality of the $\mathcal{C}$-semicompletion $H$ implies that $H=H\cap \breve{G}\subseteq 
\breve{G}$ and, finally, $H=\mathcal{C}[G)$. 

Having defined the $\mathcal{C}$-semicompletions of paratopological groups, we can pose 
the following two general problems.

\begin{problem}\label{Pr1} 
Explore the categorial properties of $\mathcal{C}$-semicompletions.
\end{problem}

\begin{problem}\label{Pr2}
Characterize the paratopological groups $G$ for which the $\mathcal{C}$-semicompletion 
$\mathcal{C}[G)$ of $G$ is Hausdorff.
\end{problem}

Some partial answers to these problems will be given in Sections~\ref{s1} and~\ref{s2}.

Again, let $\C$ be a class of continuous homomorphisms of paratopological groups. 
Now we introduce the notion of a $\mathcal{C}$-completion of a paratopological group. 
A paratopological group $G$ is {\em $\mathcal{C}$-complete} if $G$ is a $\mathcal{C}$-semicompletion 
of $G$. Equivalently, a paratopological group $G$ is {\em $\mathcal{C}$-complete} if each homomorphism $h\colon X\to G$ in the class $\mathcal{C}$ extends to a continuous homomorphism $\breve{h}\colon
\breve{X}\to G$ defined on the Ra\u{\i}kov completion of $X$. By the extension property of Ra\u{\i}kov completions, each $\sharp$-complete paratopological group $G$ (hence, $\breve{G}$) is 
$\mathcal{C}$-complete. In particular, the Sorgenfrey line $\mathbb{S}$ is $\mathcal{C}$-complete.

By the {\em $\mathcal{C}$-completion}, $\mathcal{C}[G]$, of a paratopological group $G$ we understand 
the intersection $\bigcap\mathcal{S}$ of the family $\mathcal{S}$ of all $\mathcal{C}$-complete subgroups $H\subseteq\breve{G}$ that contain $G$. The Hausdorffness of the groups coreflections of paratopological groups implies that the $\mathcal{C}$-completion $\mathcal{C}[G]$ of any paratopological group $G$ is 
$\mathcal{C}$-complete and, hence, $\mathcal{C}[G)\subseteq \mathcal{C}[G]$.

It is clear from the above definitions that if $\mathcal{C}$ is the empty class (or, more generally,
every homomorphism $h\colon X\to G$ with $h\in \mathcal{C}$, if exists, is trivial), then $G$ is 
$\mathcal{C}$-complete and, therefore, $G=\mathcal{C}[G)=\mathcal{C}[G]$.

\begin{problem}\label{Pr3} 
Explore the categorial properties of the operation of $\mathcal{C}$-completion in the category 
of paratopological groups and their continuous homomorphisms.
\end{problem}

\begin{problem}\label{Pr4}
Characterize the paratopological groups $G$ such that the $\mathcal{C}$-completion $\mathcal{C}[G]$ 
of $G$ is Hausdorff.
\end{problem}

\begin{problem}\label{Pr5}
For which paratopological groups (and classes $\mathcal{C}$ of homomorphisms) do their 
$\mathcal{C}$-semicompletions coincide with their $\mathcal{C}$-completions?
\end{problem}

\section{Categorial properties of $\mathcal{C}$-completions and $\mathcal{C}$-semicompletions}\label{s1}

A class $\mathcal{C}$ of continuous homomorphisms between paratopological groups is called {\em composable} if  for any homomorphism $f\colon X\to Y$ in the class $\mathcal{C}$ and any continuous homomorphism $g\colon Y\to Z$ of paratopological groups, the composition $g\circ f$ is in
$\mathcal{C}$. For example, for any class $\mathcal{P}$ of paratopological groups, the class 
$\mathcal{C}$ of continuous homomorphisms $h\colon X\to Y$ between paratopological groups with 
$X\in\mathcal{P}$ is composable.

The following proposition shows that for a composable class $\mathcal{C}$, the constructions of 
$\mathcal{C}$-semicompletion and $\mathcal{C}$-completion are functorial in the category of 
paratopological groups and their continuous homomorphisms.

\begin{proposition}\label{Prop:1}
Let $\mathcal{C}$ be a composable class of continuous homomorphisms of paratopological 
groups. For any continuous homomorphism $h\colon X\to Y$ of paratopological groups, its 
continuous extension $\breve{h}\colon\breve{X}\to\breve{Y}$ satisfies $\breve{h}[\mathcal{C}[X)]
\subseteq \mathcal{C}[Y)$ and $\breve{h}[\mathcal{C}[X]]\subseteq \mathcal{C}[Y]$.
\end{proposition}

\begin{proof} 
To see that $\breve{h}[\mathcal{C}[X)]\subseteq \mathcal C[Y)$, it suffices to check that the preimage 
$\breve{h}^{-1}[\mathcal C[Y)]$ is a $\C$-semicompletion of $X$. Given any homomorphism 
$f\colon Z\to X$ in the class $\C$, consider its continuous homomorphic extension $\breve{f}\colon
\breve{Z}\to\breve{X}$. Then $\breve{h}\circ\breve{f}\colon\breve{X}\to \breve{Y}$ is a continuous extension 
of the homomorphism $h\circ f\colon Z\to Y$. Taking into account that $\C[Y)$ is a $\C$-semicompletion of 
$Y$ and the topology of the group reflection of $\breve{Y}$ is Hausdorff, we conclude that $(\breve{h}\circ 
\breve{f})[\breve{Z}]\subseteq\C[Y)$ and hence $\breve{f}[\breve{Z}]\subseteq \breve{h}^{-1}[\C[Y)]$. This implies that $\breve{h}^{-1}[\C[Y)]$ is a $\C$-semicompletion of $X$ and $\C[X)\subseteq \breve{h}^{-1}[\C[Y)]$, by the minimality of $\C[X)$.

Next, we show that the preimage $\breve{h}^{-1}[\C[Y]]$ is a $\C$-complete paratopological group. Given 
any homomorphism $f\colon Z\to \breve{h}^{-1}[\C[Y]]$ in the class $\C$, consider its unique continuous extension $\breve{f}\colon \breve{Z}\to \breve{X}$. Then $\breve{h}\circ\breve{f}\colon\breve{Z}\to\breve{Y}$ 
is a unique continuous extension of the homomorphism $h\circ f\colon Z\to \C[Y]$. Taking into account that 
the paratopological group $\C[Y]$ is $\C$-complete and the topology of the group reflection of $\breve{Y}$ 
is Hausdorff, we conclude that $(\breve{h}\circ \breve{f})[\breve{Z}]\subseteq\C[Y]$, which implies the 
inclusion $\breve{f}[\breve{Z}]\subseteq \breve{h}^{-1}[\C[Y]]$. Therefore, the paratopological group 
$\breve{h}^{-1}[\C[Y]]$ is $\C$-complete and $\C[X]\subseteq\breve{h}^{-1}[\C[Y]]$, by the minimality 
of $\C[X]$. 
\end{proof}

\section{The Hausdorff property of $\mathcal{C}$-completions of paratopological groups}\label{s2}
In this section, we present some necessary and some sufficient conditions on a paratopological
group $G$ in order that the $\mathcal{C}$-semicompletion $\mathcal{C}[G)$ or $\C$-completion 
$\mathcal{C}[G]$ of $G$ be Hausdorff. Let us recall that for a regular paratopological group $G$, 
its Ra\u\i kov completion $\breve G$ is regular and so are the paratopological groups $\C[G)$ and 
$\C[G]$.

\begin{proposition}\label{p1} 
The $\mathcal{C}$-completion $\mathcal{C}[G]$ of a paratopological group $G$ is Hausdorff if 
and only if $G$ is a subgroup of a Hausdorff $\mathcal{C}$-complete paratopological group.
\end{proposition}

\begin{proof} 
The ``only if'' part is trivial. To prove the ``if'' part, assume that $G$ is a subgroup of a Hausdorff 
$\mathcal{C}$-complete paratopological group $H$. Taking into account that $G\subseteq H$ 
and $G^\sharp\subseteq H^\sharp$, we conclude that $\breve{G}\subseteq\breve{H}$. The 
$\mathcal{C}$-completeness of the paratopological groups $\breve{G}$ and $H$ implies that 
the paratopological group $\breve{G}\cap H$ is $\mathcal{C}$-complete and hence 
$\mathcal{C}[G]\subseteq \breve{G}\cap H$ is Hausdorff, being a subgroup of the Hausdorff 
paratopological group $H$.
\end{proof}

Let $\mathcal{P}$ be a property. 
A paratopological group $H$ is said to be \emph{projectively 
$\mathcal{P}$} if every neighborhood of the identity element in $H$ contains a closed invariant 
(equivalently, ``normal'' in the algebraic sense) 
subgroup $N$ such that the quotient paratopological group $H/N$ has $\mathcal{P}$. Similarly,
a paratopological group $X$ is said to be {\em projectively $\Psi$} if for every neighborhood $U$ 
of the identity in $X$, there exists a continuous homomorphism $h\colon X\to Y$ to a Hausdorff 
paratopological group of countable pseudocharacter such that $h^{-1}(e_Y)\subseteq U$; here 
$e_Y$ denotes the identity of the group $Y$. 

It is easy to see that the projectively $\Psi$ 
paratopological groups are exactly the projectively $\mathcal{P}$ groups, where $\mathcal{P}$ 
is the property of being a Hausdorff space of countable pseudocharacter. Indeed, let $G$ be a 
projectively $\Psi$ paratopological group and $U$ be a neighborhood of the identity in $G$.
By our assumption, there exists a continuous homomorphism $h\colon G\to H$ onto a Hausdorff 
paratopological group $H_\alpha$ of countable pseudocharacter such that $h^{-1}(e_H)\subseteq U$, 
where $e_H$ is the identity of $H$. Let $N=h^{-1}(e_H)$ be the kernel of $h$ and $p\colon G\to G/N$ 
be the quotient homomorphism. Clearly there exists a continuous one-to-one homomorphism 
$j\colon G/N\to H$ satisfying $h=j\circ p$. Thus $j$ is a continuous bijection of $G/N$ onto $H$ 
and, hence, the quotient group $G/N$ is Hausdorff and has countable pseudocharacter. Therefore, 
$G$ is projectively $\mathcal{P}$. The inverse implication is evident. 

According to \cite[Proposition~2]{ST} every projectively Hausdorff paratopological group is
Hausdorff. In particular, all projectively $\Psi$ paratopological groups are Hausdorff. This
fact also follows from the characterization of projectively $\Psi$ groups presented in the
next lemma.

\begin{lemma}\label{Le:Evid}
A paratopological group $G$ is projectively $\Psi$ if and only if $G$ is topologically isomorphic
 to a subgroup of a product of Hausdorff paratopological groups of countable 
pseudocharacter.
\end{lemma}

\begin{proof}
The sufficiency is evident, so we verify only the necessity. Assume that $G$ is a projectively 
$\Psi$ paratopological group. Applying \cite[Proposition~2]{ST} we conclude that $G$ is 
Hausdorff. Let $\{U_\alpha: \alpha\in A\}$ be a neighborhood base at the identity $e$ of $G$. 
For every $\alpha\in A$, take an open neighborhood $V_\alpha$ of $e$ in $G$ such that
$V_\alpha^2\subset U_\alpha$. By our assumption, there exists a continuous homomorphism 
$f_\alpha$ of $G$ onto a Hausdorff paratopological group $H_\alpha$ of countable pseudocharacter 
such that $f_\alpha^{-1}(e_\alpha)\subseteq V_\alpha$, where $e_\alpha$ is the identity of $H_\alpha$. 
Let $N_\alpha$ be the kernel of $f_\alpha$ and $p_\alpha\colon G\to G/N_\alpha$ be the quotient
homomorphism. Clearly there exists a continuous one-to-one homomorphism $j_\alpha\colon
G/N_\alpha\to H_\alpha$ satisfying $f_\alpha=j_\alpha\circ p_\alpha$. Thus $j_\alpha$ is a
continuous bijection of $G/N_\alpha$ onto $H_\alpha$ and, hence, the quotient group $G/N_\alpha$
is Hausdorff and has countable pseudocharacter. 

Denote by $p$ the diagonal product of the family $\{p_\alpha: \alpha\in A\}$. Then $p$ is a continuous 
homomorphism of $G$ to $P=\prod_{\alpha\in A} G/N_\alpha$. We claim that $p$ is an isomorphic 
topological embedding. First, the kernel of $p$ is trivial. Indeed, if $x\in G$ and $x\neq e_G$, take a
neighborhood $U$ of $e_G$ such that $x\notin U$. There exists $\alpha\in A$ such that 
$p_\alpha^{-1}(e_\alpha)\subseteq U_\alpha \subseteq U$, whence it follows that $p_\alpha(x)\neq 
e_\alpha$. Hence $p(x)\neq e_H$ and $p$ is injective.

Further, let $U$ be an arbitrary neighborhood of $e$ in $G$. Take $\beta\in A$ with 
$U_\beta\subseteq U$. Then $f_\beta^{-1}f_\beta(V_\beta)=V_\beta N_\beta\subseteq
V_\beta^2\subseteq U_\beta\subseteq U$. Since $j_\beta$ is one-to-one, it follows from 
$f_\beta=j_\beta\circ p_\beta$ that $p_\beta^{-1}p_\beta(V_\beta)=
f_\beta^{-1}f_\beta(V_\beta)\subseteq U$. Thus the open neighborhood $p_\beta(V_\beta)$
of the identity in $H_\beta$ satisfies $p_\beta^{-1}p_\beta(V_\beta)\subseteq U$. Denote 
by $\pi_\beta$ the projection of $\prod_{\alpha\in A} G/N_\alpha$ to the factor $G/N_\beta$.
Applying the equality $p_\beta=\pi_\beta\circ p$ we deduce that the open neighborhood 
$W=p(G)\cap \pi_\beta^{-1} p_\beta(V_\beta)$ of the identity in $p(G)$ satisfies 
$p^{-1}(W)=p_\beta^{-1}(p_\beta(V_\beta))\subseteq U$. We have thus proved that for 
every neighborhood $U$ of the identity in $G$, there exists a neighborhood $W$ of
the identity in $p(G)$ such that $p^{-1}(W)\subseteq U$. This property of the continuous
monomorphism $p$ implies that $p$ is an isomorphic topological embedding of $G$ 
into $\prod_{\alpha\in A} G/N_\alpha$ (see also \cite[Section~3.4]{AT}).
\end{proof}

A paratopological group $G$ is called \emph{$\w$-balanced} if for any neighborhood $U$ of the 
identity in $G$, there exists a countable family $\mathcal V$ of neighborhoods of the identity in 
$G$ such that for each $x\in G$, one can find $V\in\mathcal V$ satisfying $xVx^{-1}\subseteq U$.

The \emph{Hausdorff number} $Hs(G)$ of a Hausdorff paratopological group $G$ is the smallest  
cardinal $\kappa\geq 1$ such that for every neighborhood $U\subseteq G$ of the identity in $G$, 
there exists a family $\{V_\alpha\}_{\alpha\in\kappa}$ of neighborhoods of the identity in $G$ such 
that $\bigcap_{\alpha\in\kappa}V_\alpha V_\alpha^{-1}\subseteq U$. The Hausdorff number was 
introduced and studied in \cite{Tka09}.

According to \cite[Theorem~2.7]{Tka09}, every $\w$-balanced paratopological group $G$ with
$Hs(G)\leq\omega$ is a subgroup of a Tychonoff product of first-countable Hausdorff paratopological 
groups. By Lemma~\ref{Le:Evid}, the latter implies that the $\w$-balanced paratopological groups with
countable Hausdorff number are projectively $\Psi$.

\begin{theorem}\label{Th:3.3}
Let $\mathcal{C}$ be a subclass of the class of continuous homomorphisms from pseudocompact 
topological groups to paratopological groups. If a paratopological group $G$ is projectively $\Psi$, 
then its $\mathcal{C}$-completion $\mathcal{C}[G]$ is Hausdorff. If $G$ is $\w$-balanced and satisfies 
$Hs(G)\le \w$, then the paratopological group $\mathcal{C}[G]$ is $\w$-balanced, Hausdorff, 
and  satisfies $Hs(\mathcal{C}[G])\le\w$.
\end{theorem}

\begin{proof} 
Let $G$ be a projectively $\Psi$ paratopological group. By Lemma~\ref{Le:Evid}, $G$ is 
topologically isomorphic to a subgroup of a Tychonoff product $Y=\prod_{\alpha\in A}Y_\alpha$ 
of Hausdorff paratopological groups of countable pseudocharacter. We claim that the paratopological 
group $Y$ is $\C$-complete. Consider any homomorphism $f\colon Z\to Y$ in the class $\C$. By 
the choice of $\C$, $Z$ is a pseudocompact topological group. For every $\alpha\in A$, denote by 
$\pi_\alpha\colon Y\to Y_\alpha$ the natural projection and consider the continuous homomorphism $f_\alpha=\pi_\alpha\circ f\colon Z\to Y_\alpha$. Let $e_\alpha$ be the identity of the group 
$Y_\alpha$. Clearly $Z_\alpha=f_\alpha^{-1}(e_\alpha)$ is a closed invariant subgroup of $Z$
and the quotient topological group $Z/Z_\alpha$ admits a continuous injective homomorphism 
to the group $Y_\alpha$ of countable pseudocharacter. Hence $Z/Z_\alpha$ has countable
pseudocharacter as well and we can apply \cite[Proposition~2.3.12]{Tka18} to conclude that 
$Z/Z_\alpha$ is a compact metrizable group. 
Let $q_\alpha\colon Z\to Z/Z_\alpha$ be the quotient homomorphism and $g_\alpha\colon 
Z/Z_\alpha\to Y_\alpha$ be a unique continuous injective homomorphism such that 
$f_\alpha=g_\alpha\circ q_\alpha$. Since the compact topological group $Z/Z_\alpha$ is 
Ra\u{\i}kov complete, the quotient homomorphism $q_\alpha\colon Z\to Z/Z_\alpha$ admits 
a continuous homomorphic extension $\breve{q}_\alpha\colon \breve{Z}\to Z/Z_\alpha$. Then 
$\breve{f}_\alpha= g_\alpha\circ\breve{q}_\alpha\colon \breve{Z}\to Y_\alpha$ is a continuous 
extension of the homomorphism $f_\alpha$.

The diagonal product of the homomorphisms $\breve{f}_\alpha\colon \breve{Z}\to Y_\alpha$ with 
$\alpha\in A$ is a continuous homomorphism $\breve{f}\colon \breve{Z}\to Y=\prod_{\alpha\in A}Y_\alpha$ extending the homomorphism $f$ and witnessing that the paratopological group $Y$ is $\C$-complete. 
By Proposition~\ref{p1}, the $\C$-completion $\C[X]$ of $X$ is Hausdorff. 

Now assume that the paratopological group $G$ is $\w$-balanced and satisfies $Hs(G)\le\w$. By \cite[Theorem~2.7]{Tka09}, $G$ is a subgroup of a Tychonoff product $Y=\prod_{\alpha\in A}Y_\alpha$ 
of first-countable Hausdorff paratopological groups. Hence $G$ is projectively $\Psi$, by Lemma~\ref{Le:Evid}.
By the above argument, the paratopological group $Y$ is $\C$-complete and, hence, $\C[G]$ can be 
identified with a subgroup of $Y$. By Propositions~2.1--2.3 in \cite{Tka09}, the subgroup $\C[G]$ of the Tychonoff product $Y$ of first-countable Hausdorff paratopological groups is $\w$-balanced and satisfies $Hs(\C[G])\le\w$.
\end{proof}

A topological group $G$ is called {\em precompact} if its Ra\u{\i}kov completion $\rho{G}$ is compact. 
This happens if and only if for any neighborhood $U$ of the identity in $G$, there exists a finite set 
$F\subseteq G$ such that $G=UF=FU$.

\begin{proposition}\label{p:precompact} 
Let $\mathcal C$ be a subclass of the class of continuous homomorphisms from 
precompact topological groups to paratopological groups. A Hausdorff paratopological group 
$G$ is $\mathcal{C}$-complete if and only if for any homomorphism $h\colon X\to G$ in the 
class $\C$, the image $h[X]$ has compact closure in $G$.
\end{proposition}

\begin{proof} 
\emph{Sufficiency.} Take any homomorphism $h\colon X\to G$ in the class $\C$ and assume that the 
image $h[X]$ has compact closure $\overline{h[X]}$ in $G$. It follows from \cite[Proposition~1.4.10]{AT}  
that $H=\overline{h[X]}$ is a Hausdorff compact topological semigroup and $H$, being a subsemigroup 
of a group, is cancellative. By Numakura's theorem (see \cite{Num} or \cite[Theorem~2.5.2]{AT}), 
$H$ is a compact topological group. Since $H$ is Ra\u{\i}kov complete, the homomorphism 
$h\colon X\to H$ has a continuous homomorphic extension $\breve{h}\colon\breve{X}\to H\subseteq G$, witnessing that the paratopological group $G$ is $\C$-complete.\smallskip

\emph{Necessity.} Assume that the paratopological group $G$ is $\C$-complete. Then every homomorphism 
$h\colon X\to G$ in the class $\C$ has a continuous homomorphic extension $\breve{h}\colon\breve{X}\to G$ 
to the Ra\u{\i}kov completion $\breve{X}=\rho{X}$ of the topological group $X$. The precompactness of $X$ 
guarantees that the topological group $\breve{X}$ is compact and so is its image $\breve{h}[\breve{X}]$ 
in $G$. Since the space $G$ is Hausdorff, the compact subspace $\breve{h}[\breve{X}]$ is closed in $G$ 
and  the closure $\overline{h[X]}$ of $h[X]$ in $G$, being a closed subset of the compact space 
$\breve{h}[\breve{X}]$, is compact.   
\end{proof}

A subset $F$ of a topological space $X$ is called \emph{functionally closed} (or else a 
\emph{zero-set}) if $F=f^{-1}(0)$ for some continuous function $f\colon X\to \IR$. A paratopological 
group $G$ is said to be \emph{simply $sm$-factorizable} if for every functionally closed set $A$ in $G$, 
there exists a continuous homomorphism $h\colon G\to  H$ onto a separable metrizable paratopological 
group $H$ such that $A=h^{-1}[B]$, for some closed set $B\subseteq H$ (see \cite[Definition~5.6]{AT}). 

A subspace $X$ of a topological space $Y$ is called \emph{$C$-embedded} in $Y$ if each 
continuous real-valued function on $X$ has a continuous extension over $Y$.

\begin{proposition} \label{Prop:3.5}
Every regular simply $sm$-factorizable paratopological group $G$ is $C$-embedded in its 
$\mathcal{C}$-completion $\mathcal{C}[G]$ provided $\mathcal{C}$ is a subclass of the class 
of continuous homomorphisms of pseudocompact topological groups to paratopological groups.
\end{proposition}

\begin{proof} 
The space $G$ is Tychonoff, by \cite[Corollary~5]{BRT}. We apply \cite[Theorem~4.3]{XT} 
according to which the realcompactification $\upsilon{G}$ of the space $G$ admits the structure 
of paratopological group containing $G$ as a dense paratopological subgroup. Since the closure 
of every pseudocompact subspace in $\upsilon{G}$ is compact and every pseudocompact 
topological group is precompact, the paratopological group $\upsilon{G}$ is $\C$-complete 
according to Proposition~\ref{p:precompact}. Then $G\subseteq \C[G]\subseteq \upsilon{G}$, 
which implies that $G$ is $C$-embedded in $\C[G]$, being $C$-embedded in $\upsilon{G}$.
\end{proof}

Now we present a necessary condition on a paratopological group $G$ for the Hausdorffness 
of its $\C$-semicompletion $\mathcal{C}[G)$. In view of the inclusion $\mathcal{C}[G)\subseteq
\mathcal{C}[G]$, the same condition is necessary for the Hausdorffness of $\mathcal{C}[G]$.

\begin{proposition} 
Let $\mathcal{C}$ be a subclass of the class of continuous homomorphisms from precompact 
topological groups to paratopological groups. If the $\C$-semicompletion $\C[G)$ of $G$ is 
Hausdorff, then for any homomorphism $h\colon X\to G$ in $\C$, the closure of $h[X]$ in $G$ 
is a precompact topological group.
\end{proposition}

\begin{proof} 
Let $h\colon X\to G$ be any homomorphism in the class $\C$. By the definition of a $\C$-semicompletion, 
the homomorphism $h$ has a continuous extension $\breve{h}\colon\breve{X}\to \C[G)$. The precompactness of $X$ ensures that $\breve{X}=\rho{X}$ is a compact topological group. Assume that the paratopological 
group $\C[G)$ is Hausdorff. Then the compact subspace $\breve{h}[\breve{X}]$ of $\C[G)$ is closed. It 
follows that $\breve{h}[\breve{X}]$ is a compact topological group that is topologically isomorphic to a 
quotient group of the compact topological group $\breve{X}$. Then the closure of $h[X]$ in $G$ is a 
precompact topological group, which is equal to the intersection $G\cap \breve{h}[\breve{X}]$ of the 
group $G$ with the compact topological group $\breve{h}[\breve{X}]\subseteq\C[G)$.  
\end{proof} 

Finally we present an example of a Hausdorff paratopological group $G$ with a non-Hausdorff $\C$-semicompletion, where $\C$ is the class of all continuous homomorphisms from sequentially compact topological groups to paratopological groups. As usual, a space $X$ is \emph{sequentially compact} 
if each sequence in $X$ contains a convergent subsequence. Also, $X$ is said to be \emph{$\omega$-bounded} if the closure in $X$ of every countable set is compact. 

\begin{example}\label{Ex:old} 
Let $\mathcal{C}$ be the class of continuous homomorphisms from sequentially compact topological 
groups to paratopological groups. There exists a Hausdorff paratopological abelian group $G$ whose 
$\mathcal{C}$-semicompletion $\mathcal{C}[G)$ fails to be a $T_1$-space. In addition, $G$ contains 
a subgroup $H$ such that $H$ is a sequentially compact $\omega$-bounded topological group but its 
closure in $G$ is not a group.
\end{example}

\begin{proof}
Let $\mathbb T=\{z\in\mathbb C: |z|=1\}$ be the torus group with its usual topology and multiplication 
inherited from the complex plane. So the identity of $\mathbb T$ is $1$. Denote by $\Sigma$ the 
subgroup of the Tychonoff product $\mathbb T^{\omega_1}$ defined as follows:
$$
\Sigma= \{x\in \mathbb T^{\omega_1}:  |\{\alpha\in\omega_1: x_\alpha\neq 1\}|\leq\omega\}.
$$
According to Corollaries~1.6.33 and~1.6.34 of \cite{AT}, the space $\Sigma$ is Fr\'echet-Urysohn 
and $\omega$-bounded. Hence $\Sigma$ is sequentially compact.

Take an element $c\in\mathbb T^{\omega_1}$ of infinite order such that $\langle c\rangle\cap\Sigma=\{e\}$, where $e$ is the identity element of $\mathbb T^{\omega_1}$ and $\langle c\rangle$ is the cyclic group generated by $c$. For every open neighborhood $U$ of $e$ in $\mathbb T^{\omega_1}$, we define a subset $O_U$ of $G=\Sigma\times\mathbb Z$ by letting
$$
O_U= \bigcup_{n\in\omega} \big(c^n U\cap \Sigma\big)\times\{n\}.
$$
Here $\omega$ is identified with the subset $\{0,1,2,\ldots\}$ of $\mathbb Z$. A routine verification 
shows that the sets $O_U$ with $U$ as above, constitute a base at the identity $(e,0)\in G$ for a 
Hausdorff paratopological group topology on $G$. 
	
It turns out that the subgroup $H=\Sigma\times\{0\}$ of $G$ has the required properties. Let us show that 
the closure of $H$ of $G$ is the set $\Sigma\times (-\w)$, where $-\w=\{0,-1,-2,\ldots\}$, so this closure is 
not a subgroup of $G$. First, we claim that any element $g=(x,k)\in G$ with $k>0$ is not in the closure of 
$H$.  Indeed, let $U$ be an arbitrary open neighborhood of $e$ in $\mathbb T^{\omega_1}$. Then the set 
$g O_U$ is an open neighborhood of $g$ in $G$ disjoint from $H$. 
	
Further, consider an element $g=(x,-k)\in G$, where $x\in\Sigma$ and $k\in\omega$. If $k=0$, then $g\in\Sigma\times\{0\}=H$. Assume that $k>0$ and take a basic open neighborhood $O_U$ of the 
identity in $G$, where $U$ is an open neighborhood of $e$ in $\mathbb T^{\omega_1}$. Then 
$$
g\cdot O_U= g\cdot\bigcup_{n\in\omega} \big((c^nU\cap\Sigma)\times\{n\}\big) =
\bigcup_{n\in\omega} (xc^nU\cap\Sigma)\times \{n-k\}.
$$
Therefore, $g\cdot O_U\cap H= g\cdot O_U\cap (\Sigma\times\{0\})= 
(xc^k U\cap\Sigma)\times \{k\}\neq\emptyset$. Hence $g$ is in the closure 
of $H$. We have thus shown that the closure of $H$ in $G$ is the asymmetric
subset $\Sigma\times (-\w)$ of $G$, which implies the second claim of the example.

The definition of the topology of the paratopological group $G$ implies that the topology of the group coreflection $G^\sharp$ coincides with the product topology on $\Sigma\times \IZ$ and then the Ra\u{\i}kov completion $\rho{G}^\sharp$ can be identified with the product $\IT^{\w_1}\times\IZ$, where $\IZ$ is endowed 
with the discrete topology. By \cite[\S2.2]{BR}, a neighborhood base for the topology of the Ra\u{\i}kov completion $\breve{G}$ at its identity consists of the closures $\overline{O_U}$ of the basic neighborhoods $O_U$ in the product space $\IT^{\w_1}\times\IZ$. It is easy to see that for any neighborhood $U$ of $e$ 
in $\IT^{\w_1}$, the closure $\overline{O_U}$ of $O_U=\bigcup_{n\in\w}(c^nU\cap\Sigma)\times\{n\}$ coincides with the set $\bigcup_{n\in\w}(c^n\overline{U})\times\{n\}$. We see, therefore, that each neighborhood of the identity in $\breve{G}$ contains the point $(c,1)$, which means that the Ra\u{\i}kov completion $\breve{G}$ of $G$ does not satisfy the $T_1$ separation axiom.

It remains to prove that $\breve{G}$ is the $\C$-semicompletion of $G$. Since the subgroup $\Sigma$ of 
$\IT^{\w_1}$ is sequentially compact, the continuous homomorphism 
$$
h\colon \Sigma\to \Sigma\times\{0\}\subseteq G,\quad h: x\mapsto \langle x,0\rangle,
$$ 
extends to a continuous homomorphism $\bar h\colon \IT^{\w_1}\to \C[G)\subseteq\breve{G}$. Since 
$\IT^{\w_1}$ is a topological group, the homomorphism $\bar h$ remains continuous with  respect to the topology of the Ra\u{\i}kov completion $\rho{G}^\sharp=\IT^{\w_1}\times\IZ$ of the topological group $G^\sharp$. Now the compactness of $\IT^{\w_1}$ implies that $\IT^{\w_1}\times\{0\}=\bar h[\IT^{\w_1}]\subseteq \C[G)$ and, hence, $\C[G)=\breve{G}$.
\end{proof}

\end{document}